\documentclass{amsart}
\usepackage{amsthm,amsmath,amsfonts,amssymb,mathtools}
\usepackage{color,graphicx,hyperref,nicefrac}
\usepackage{tikz}
\newtheorem{theorem}{Theorem}[section]
\newtheorem{lemma}[theorem]{Lemma}
\theoremstyle{definition}
\newtheorem{definition}[theorem]{Definition}
\theoremstyle{remark}
\newtheorem{remark}[theorem]{Remark}
\numberwithin{equation}{section}

\newcommand{\bs}{\boldsymbol}
\newcommand{\kernel}{\kappa}
\renewcommand{\d}{\operatorname{d}\!}
\newcommand{\spn}{\operatorname{span}}
\newcommand{\iso}{\operatorname{iso}}
\newcommand{\mix}{\operatorname{mix}}
\newcommand{\isomix}{\operatorname{iso-mix}}
\begin{document}
\title[Kernel interpolation in Sobolev spaces of hybrid regularity]
{Kernel interpolation in Sobolev spaces of hybrid regularity}
\author{Michael Griebel}
\address{Michael Griebel,
Institut f\"ur Numerische Simulation,
Universit\"at Bonn, Friedrich-Hirzebruch-Allee 7, 53115 Bonn, Germany
and
Fraunhofer Institute for Algorithms and Scientific Computing (SCAI), 
Schloss Birlinghoven, 53754 Sankt Augustin, Germany
}
\email{griebel@ins.uni-bonn.de}
\author{Helmut Harbrecht}
\address{
Helmut Harbrecht,
Departement Mathematik und Informatik,
Universit\"at Basel, Spiegelgasse 1, 4051 Basel, Switzerland
}
\email{helmut.harbrecht@unibas.ch}
\date{\today}

\subjclass[2020]{Primary 41A46; Secondary 41A63, 46E35}
\keywords{Reproducing kernel Hilbert space, optimized sparse grid, 
error bounds, cost}

\begin{abstract}
Kernel interpolation in tensor product reproducing
kernel Hilbert spaces allows for the use of sparse grids
to mitigate the curse of the dimension. Typically, besides
the generic constant, only a dimension dependent power 
of a logarithm term enters here into complexity estimates. 
We show that optimized 
sparse grids can avoid this logarithmic factor when the 
interpolation error is measured with respect to Sobolev 
spaces of hybrid regularity. Consequently, in such a situation, 
the complexity of kernel interpolation does not suffer 
from the curse of dimension.
\end{abstract}
\maketitle

\section{Introduction}
Kernel interpolation is based on the theoretical 
framework provided by {\em reproducing kernel 
Hilbert spaces} \cite{Aronszajn}, RKHS for short. 
An RKHS is a specific type of Hilbert space of 
functions where every function's 
value at any given point can be {\em reproduced} via 
its inner product with the {\em reproducing kernel}. 
Mathematically, the reproducing kernel is the Riesz 
representer of the point evaluation. This feature 
amounts to a simple and efficient way to obtain
kernel-based approximations that interpolate given 
scattered data. To this end, the representer theorem
\cite{Wahba2} is used, which states that the interpolant 
can be written as a finite linear combination of the kernel 
function evaluated at the data points. Kernel interpolation arises 
in machine learning and scattered data approximation, compare 
\cite{Fasshauer,HSS08,Schaback2,Wendland,Williams,Rasmussen}. 
The quality of this approximation, i.e., the estimate of the 
approximation error, is well established and can be found 
in e.g.~\cite{Fasshauer,Wendland}.

In this article, we aim at the construction of \emph{optimized sparse 
grids} for the dimension robust interpolation with respect to tensor 
product kernels. A fundamental contribution to sparse grid kernel 
interpolation has recently been provided by \cite{Kempf1,Kempf2}. 
While the sparse grid construction therein relies on a multilevel 
approach invoking level dependent correlation lengths of the 
kernel function under consideration, we will use here a kernel 
function of \emph{fixed correlation length} and construct for that 
the sparse grid interpolant. This methodology has already been 
exploited and analyzed in \cite{GHM}, including so-called 
superconvergence, compare \cite{Schaback1,Sloan}. 

Error estimates for kernel interpolation have been established
so far only in classical, isotropic Sobolev spaces or in Sobolev
spaces of dominating mixed derivatives, cf.~\cite{GHM,Kempf1,Wendland}.
The resulting approximation rates are in general not independent 
of the underlying dimension. To remove this dependency, 
we consider in this article now so-called Sobolev spaces 
of hybrid regularity. This class of function spaces and their 
associated norms are designed to simultaneously capture 
two different types of regularity -- isotropic smoothness 
and mixed smoothness. They play an important role in 
quantum chemistry as they describe the regularity of 
the Coulombic wavefunction in the electronic Schr\"odinger 
equation, see \cite{Ayers,Meng,Harry1} for example. As 
shown in \cite{GK1,GK2}, one can derive optimized 
sparse grids for such spaces and indeed obtain 
dimension robust approximation rates under certain 
circumstances. Further applications of Sobolev spaces 
of hybrid regularity can be found in \cite{GOV} for 
parabolic initial boundary value problems and in
\cite{HS} for homogenization problems.

The analysis in \cite{GK1,GK2} is based on wavelet 
bases for the definition of the sparse grid spaces in 
combination with norm equivalences that hold between 
a whole scale of Sobolev spaces. The same technique 
applies also to the Fourier transform, compare \cite{GH1,
GH2}. In contrast, in the present article, we show how 
these results can be transferred to kernel interpolation
where one has \emph{no} norm equivalence. Indeed, to 
verify dimension robust approximation rates, it suffices to 
have Jackson and Bernstein inequalities, which are available 
for kernel interpolation. Moreover, in contrast to the approach
in \cite{GK1,GK2}, our new technique allows to substantially 
extend the range of the scales of Sobolev spaces where 
the estimates are valid. In this sense, we implicitly improve 
upon the results in \cite{GK1,GK2}.

For sake of simplicity and clearness of representation, we 
restrict ourselves to the simple setting of tensor products 
of Sobolev spaces of periodic functions on the interval 
$\mathbb{T} = [0,1]$ and to equidistant point distributions. 
Hence, we consider the $d$-dimensional torus $\mathbb{T}^d$ 
in higher dimensions. The extension to general product domains 
and quasi-uniform point distributions can be derived along the 
lines of \cite{GHM}. Note that we first consider the case of a 
bivariate sparse grid in order to make the ideas of the proofs 
clear to the reader. Afterwards, we extend the results to the 
multivariate setting.

The outline of this article is as follows: In Section~\ref{sec:prelim},
we introduce the univariate Sobolev spaces under consideration
and review kernel interpolation in reproducing kernel Hilbert 
spaces. In Section~\ref{sct:bivariate}, we focus on the bivariate 
situation. We introduce Sobolev spaces of hybrid regularity and 
derive related Jackson and Bernstein inequalities. Then we 
introduce optimized sparse grid spaces and provide error 
estimates for the kernel interpolation in these spaces. We extend 
the bivariate results to arbitrary dimension in Section~\ref{sct:arbitrary}. 
Concluding remarks are finally stated in Section~\ref{sec:conclusio}.

Throughout this article, in order to avoid the repeated use of 
generic but unspecified constants, we denote by $C \lesssim D$ 
that $C$ is bounded by a multiple of $D$ independently of parameters 
on which $C$ and $D$ may depend. Especially, $C \gtrsim D$ is defined 
as $D \lesssim C$, and $C \sim D$ as $C \lesssim D$ and $C \gtrsim D$.

\section{Preliminaries}\label{sec:prelim}
\subsection{The Sobolev space $\bs{H^s(\mathbb{T})}$ of periodic functions}
Let $\mathbb{T} := [0,1]$ be the unit interval and
$H^s(\mathbb{T})$ denote the periodic Sobolev space
of fractional smoothness $s\ge 0$. For $s\in\mathbb{N}$,
this space can be characterized by
\begin{align*}
  H^s(\mathbb{T}) := \big\{f\in & L^2(\mathbb{T}): 
      \|f^{(k)}\|_{L^2(\mathbb{T})}<\infty\\ 
      &\text{and}\ f^{(k-1)}(0)=f^{(k-1)}(1)
	\ \text{for all}\ k=1,\dots,s\big\},
\end{align*}
equipped with the norm
\[
  \|f\|_{H^s(\mathbb{T})}^2 = \bigg(\int_{\mathbb{T}} f(x)\d x\bigg)^2
  	+ \int_{\mathbb{T}} \big(f^{(s)}(x)\big)^2\d x <\infty.
\]
Indeed, the above expression defines a norm 
for $H^s(\mathbb{T})$ because the integrals 
of the derivatives $f^{(k)}$ vanishes for all $k=1,\dots,s-1$ 
due to the periodic boundary conditions.

We need the following definition:

\begin{definition}\label{def:RKHS}
A \emph{reproducing kernel} for a Hilbert space 
$\mathcal{H}$ of functions \(u\colon\Omega\to\mathbb{R}\)
with inner product $(\cdot,\cdot)_{\mathcal{H}}$ is a function 
$\kernel:\Omega\times\Omega\to\mathbb{R}$ such that
\begin{enumerate}
  \item $\kernel(\cdot,y)\in\mathcal{H}$ for all $y\in\Omega$,
  \item $u(y) = \big(u,\kernel(\cdot,y)\big)_\mathcal{H}$ 
  for all $u\in\mathcal{H}$ and all $y\in\Omega$.
\end{enumerate}
A Hilbert space $\mathcal{H}$ with reproducing kernel 
$\kernel\colon\Omega\times\Omega\to\mathbb{R}$ is called
\emph{reproducing kernel Hilbert space} (RKHS). 
\end{definition}

Reproducing kernels are known to be \emph{symmetric} and 
\emph{positive semidefinite}. Thereby, a continuous kernel 
$\kernel:\Omega\times\Omega\to\mathbb{R}$ is called 
positive semidefinite if 
\begin{equation}\label{eq:spd}
\sum_{i,j=1}^N \alpha_i\alpha_j
\kernel (x_i,x_j) \geq 0
\end{equation}
holds for all all mutually distinct points $x_1,\ldots,x_N\in\Omega$ and 
all $\alpha_1,\dots,\alpha_N\in\mathbb{R}$, for any $N\in\mathbb{N}$. 
The kernel is even \emph{positive definite} if the inequality in \eqref{eq:spd} 
is strict whenever at least one \(\alpha_i\) is different from $0$.

For general real $s> \frac{1}{2}$, the reproducing kernel in 
$H^s(\mathbb{T})$ is given by
\begin{equation}
\label{eq:kernel(0,1)}
  \kernel(x,y) = 1 + \sum_{k\in\mathbb{Z}\setminus\{0\}} 
  \frac{1}{(2\pi |k|)^{2s}} \exp\big(2\pi i k(x-y)\big).
\end{equation}
It is symmetric and positive definite. For 
$s\in\mathbb{N}$ being a natural number, this
kernel simplifies to
\[
 \kernel(x,y) = 1 + \frac{(-1)^{s+1}}{(2s)!} B_{2s}(|x-y|),
\]
where $B_{2s}: [0,1]\to\mathbb{R}$ denotes the Bernoulli
polynomial of degree $2s$, compare \cite{BTA,Wahba1}.

\subsection{Kernel interpolation}
We fix the $H^p(\mathbb{T})$ with $p > \frac{1}{2}$ 
as reproducing kernel Hilbert space with kernel $k(\cdot,\cdot)$
given by \eqref{eq:kernel(0,1)} for $s = p$. For $j\in\mathbb{N}_0$, we 
define the index set 
\[
\Delta_j := \{0,\ldots,2^j-1\}
\] 
and the associated 
equidistant grid 
\[
X_j := \{x_{j,k} := 2^{-j}k: k\in\Delta_j\}. 
\]
Then, the \emph{kernel interpolant}
\[
  u_j(x) = \sum_{k\in\Delta_j} u_{j,k} \kernel(x,x_{j,k})\in H^p(\mathbb{T})
\]
with respect to the grid $X_j$ is given by 
solving the linear system of equations
\begin{equation}\label{eq:system1D}
  {\bs K}_j{\bs u}_j = {\bs f}_j,
\end{equation}
where 
\begin{equation}\label{eq:1Ddiscrete}
{\bs K}_j = [\kernel(x_{j,k},x_{j,k'})]_{k,k'\in\Delta_j},\quad
{\bs u}_j = [u_{j,k}]_{k\in\Delta_j},\quad
{\bs f}_j = [u(x_{j,k})]_{k\in\Delta_j}.
\end{equation}
The matrix ${\bs K}_j$ is called \emph{kernel matrix}. It is a
periodic Toeplitz matrix, so that the linear system \eqref{eq:1Ddiscrete}
of equations can efficiently be solved by the fast Fourier transform.

The kernel interpolant is known to be the \emph{best approximation} 
of a given function $u\in H^p(\mathbb{T})$ in
\[
  V_j = \spn\{\kernel(\cdot,x):x\in X_j\}\subset H^p(\mathbb{T})
\]
with respect to the $H^p(\mathbb{T})$-norm. This
means that it solves the following Galerkin problem: 
\begin{equation}\label{eq:projectionP}
\text{Seek $P_j u\in V_j$:}\quad 
  (P_j u,v)_{H^p(\mathbb{T})} = (u,v)_{H^p(\mathbb{T})} 
  \quad\forall\ v\in V_j.
\end{equation}
In other words, the (Galerkin) projection $P_j: H^p(\mathbb{T})
\to V_j$ is $H^p(\mathbb{T})$-orthogonal. Recall that $P_j u$ is 
obtained by simply solving the associated kernel system 
\eqref{eq:system1D} with \eqref{eq:1Ddiscrete}.

\subsection{Jackson and Bernstein estimates}
The approximation $P_j u$ satisfies the error estimate
\begin{equation}\label{eq:L^2-norm}
  \|u-P_j u\|_{L^2(\mathbb{T})} \lesssim 2^{-pj}\|u\|_{H^p(\mathbb{T})}
\end{equation}
uniformly in $j\in\mathbb{N}_0$, see \cite[Proposition 11.30]{Wendland}
for example. Since there holds
\[
  (u,v)_{H^p(\mathbb{T})}
  	\lesssim \|u\|_{L^2(\mathbb{T})} \|v\|_{H^{2p}(\mathbb{T})},
\]
this rate of convergence can be doubled by using
\cite[Theorem 1]{Sloan} provided that the data satisfy
even $u\in H^{2p}(\mathbb{T})$, i.e., we then have
\begin{equation}\label{eq:doubling}
  \|u-P_j u\|_{L^2(\mathbb{T})} \lesssim 2^{-2jp}\|u\|_{H^{2p}(\mathbb{T})}.
\end{equation}
In view of this result and employing Galerkin orthogonality, we find
\begin{align*}
  \|u-P_j u\|_{H^p(\mathbb{T})}^2 &=
  	(u-P_j u,u)_{H^p(\mathbb{T})}\\
  	&\lesssim \|u-P_j u\|_{L^2(\mathbb{T})}
  	\|u\|_{H^{2p}(\mathbb{T})}\\
	&\lesssim 2^{-2jp}\|u\|_{H^{2p}(\mathbb{T})}^2,
\end{align*}
which implies
\begin{equation}\label{eq:energy-norm}
  \|u-P_j u\|_{H^p(\mathbb{T})} \lesssim 2^{-jp}\|u\|_{H^{2p}(\mathbb{T})}.
\end{equation}

Next we consider the projection $Q_j: H^p(\mathbb{T})
\to V_j$ given by 
\begin{equation}\label{eq:projectionQ}
Q_j u := P_j u - P_{j-1} u, \quad\text{where}\ P_{-1} u := 0. 
\end{equation}
We find from \eqref{eq:doubling} and \eqref{eq:energy-norm} 
the estimates
\[
  \|Q_j u\|_{L^2(\mathbb{T})}\lesssim 2^{-jp}\|u\|_{H^p(\mathbb{T})},\quad
  \|Q_j u\|_{H^p(\mathbb{T})}\lesssim 2^{-jp}\|u\|_{H^{2p}(\mathbb{T})}.
\]
By interpolation, we thus arrive at the approximation property,
also known as \emph{Jackson's inequality},
\begin{equation}\label{eq:approx}
 \|Q_j u\|_{H^{t_1}(\mathbb{T})}\lesssim 
 2^{-j(t_2-t_1)}\|u\|_{H^{t_2}(\mathbb{T})}
 \quad\text{for all $0\le t_1\le p\le t_2\le 2p$}.
\end{equation}
Finally, from \cite{NWH,SL}, we obtain for $u\in H^s(\mathbb{T})$ 
the inverse inequality, also known as \emph{Bernstein's inequality},
\[
  \|P_j u\|_{H^{t_2}(\mathbb{T})}\lesssim 2^{j(t_2-t_1)}\|P_j u\|_{H^{t_1}(\mathbb{T})}
 	\quad\text{for all $0\le t_1\le t_2 \le p$},
\]
which we will need in the following, immediately 
resulting from
\begin{equation}\label{eq:inverse}
  \|Q_j u\|_{H^{t_2}(\mathbb{T})}\lesssim 2^{j(t_2-t_1)}\|Q_j u\|_{H^{t_1}(\mathbb{T})}
 	\quad\text{for all $0\le t_1\le t_2 \le p$}.
\end{equation}

Note the different ranges where the Jackson and Bernstein 
inequalities hold. Estimate \eqref{eq:approx} is valid whenever the function 
$u$ to be approximated provides extra regularity relative to the underlying 
RKHS $H^p(\mathbb{T})$ and the error is measured in a weaker 
norm than the $H^p(\mathbb{T})$-norm. For functions $u\in H^p
(\mathbb{T})$, the inverse estimate \eqref{eq:inverse} however
bounds a stronger norm of the expression $Q_ju\in V_j$ by a 
weaker norm of this expression, but both norms are now
weaker than the $H^p(\mathbb{T})$-norm.

\section{Bivariate approximation}
\label{sct:bivariate}
\subsection{Sobolev spaces of hybrid regularity}
Sobolev spaces of hybrid regularity have been introduced
firstly in \cite{GK1}. They are given by
\begin{equation}\label{eq:GK-spaces}
  H_{\isomix}^{s,t}(\mathbb{T}^2) :=
  H^{t}(\mathbb{T})\otimes H^{s+t}(\mathbb{T})
  \cap H^{s+t}(\mathbb{T})\otimes H^{t}(\mathbb{T})
\end{equation}
This class of spaces $H_{\isomix}^{s,t}(\mathbb{T}^2)$ can be 
characterized for $t\in\mathbb{N}_0$ by
\[
H_{\isomix}^{s,t}(\mathbb{T}^2) := \big\{f\in L^2(\mathbb{T}^2):
	\|\partial^{\bs\alpha} f\|_{H_{\mix}^{t}(\mathbb{T}^2)}<\infty
	\ \text{for all}\ \|\bs\alpha\|_1\le s\big\}.
\]
Thus, the functions that are contained in $H_{\isomix}^{s,t}(\mathbb{T}^2)$ 
are all functions from the classical, isotropic Sobolev space 
$H_{\iso}^{s}(\mathbb{T}^2)$, where the $s$-th order derivatives 
provide in addition mixed Sobolev smoothness of order $t$. Note 
that the classical, isotropic Sobolev space $H_{\iso}^s(\mathbb{T}^2)$ 
satisfies
\[
 H_{\iso}^s(\mathbb{T}^2) = H_{\isomix}^{s,0}(\mathbb{T}^2)
 = H^{s}(\mathbb{T})\otimes L^2(\mathbb{T})
 \cap L^2(\mathbb{T})\otimes H^{s}(\mathbb{T})
\]
while the classical Sobolev space $H_{\mix}^t(\mathbb{T}^2)$ 
of dominating mixed regularity, also called \emph{mixed} 
Sobolev space, satisfies
\[
 H_{\mix}^t(\mathbb{T}^2) = H_{\isomix}^{0,t}(\mathbb{T}^2).
\]
We refer to Figure~\ref{fig:spaces} for an illustration of the
different Sobolev spaces under consideration. It especially
shows the obvious embeddings
\[
H_{\mix}^{s+t}(\mathbb{T}^2) \subset H_{\isomix}^{s,t}(\mathbb{T}^2)
\subset H_{\iso}^s(\mathbb{T}^2).
\]

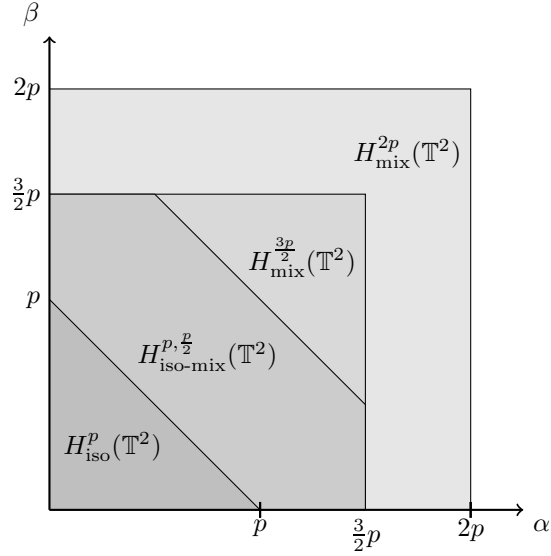
\begin{figure}[htb]
\begin{center}
\begin{tikzpicture}[scale=0.7]

\fill[gray!20] (0,0) -- (8,0) -- (8,8) -- (0,8);
\draw[-] (0,8) -- (8,8) -- (8,0); 
\node at (6.8,6.8) {$H_{\mix}^{2p}(\mathbb{T}^2)$};    

\fill[gray!30] (0,0) -- (6,0) -- (6,6) -- (0,6);
\draw[-] (0,0) -- (6,0) -- (6,6) -- (0,6); 
\node at (4.8,4.8) {$H_{\mix}^{\frac{3p}{2}}(\mathbb{T}^2)$};

\fill[gray!40] (0,0) -- (6,0) -- (6,2) -- (2,6) -- (0,6);
\draw[-] (6,0) -- (6,2) -- (2,6) -- (0,6); 
\node at (3,3) {$H_{\isomix}^{p,\frac{p}{2}}(\mathbb{T}^2)$};

\fill[gray!50] (0,0) -- (4,0) -- (0,4);
\draw[-] (0,4) -- (4,0); 
\node at (1.2,1.2) {$H_{\iso}^p(\mathbb{T}^2)$};    
      
\draw[thick,->] (0,0) -- (9,0) node[below right] {$\alpha$};
\draw[thick,->] (0,0) -- (0,9) node[above left] {$\beta$};
\node[left] at (0,4) {$p$};                  
\node[left] at (0,6) {$\frac{3}{2}p$}; 
\node[left] at (0,8) {$2p$}; 
\node[below] at (4,0) {$p$};               
\node[below] at (6,0) {$\frac{3}{2}p$}; 
\node[below] at (8,0) {$2p$}; 
\draw[thick] (8,-0.15) -- (8,0.15);          
\draw[thick] (4,-0.15) -- (4,0.15);          
\draw[thick] (8,-0.15) -- (8,0.15);          
\draw[thick] (4,-0.15) -- (4,0.15);          
\end{tikzpicture}
\end{center}
\caption{\label{fig:spaces}%
Visualization of the derivatives $\partial_x^{\alpha}\partial_y^{\beta} 
f(x,y)$ being bounded in the Sobolev spaces $H_{\isomix}^{s,t}(\mathbb{T}^2)$ 
of hybrid regularity in comparison to the isotropic Sobolev spaces 
$H_{\iso}^s(\mathbb{T}^2)$ and the Sobolev spaces of dominating 
mixed derivatives $H_{\mix}^t(\mathbb{T}^2)$ for the specific choices
of $s$ and $t$.}
\end{figure}

\subsection{Bernstein and Jackson inequalities}
If $\kernel(\cdot,\cdot)$ is the reproducing kernel in 
$H^p(\mathbb{T})$, then the product kernel
\begin{equation}\label{eq:kernel}
  \bs\kernel({\bs x},{\bs y}) := 
  	\kernel(x_1,y_1)\otimes \kernel(x_2,y_2).
\end{equation}
is the reproducing kernel in the mixed Sobolev space 
$H_{\mix}^p(\mathbb{T}^2)$. This space will serve as the 
reproducing kernel Hilbert space under consideration, 
where we carry out all our estimates in the following. 

\begin{remark}
Let $\kernel_s(x,y)$ be the kernel for the Sobolev space 
$H^{s}(\mathbb{T})$ and $\kernel_t(x,y)$ be the kernel 
for the Sobolev space $H^{t}(\mathbb{T})$, where
we assume that $d/2< t\le s$. Then, the kernel
\begin{equation}\label{eq:kernel_hybrid}
  \bs\kernel({\bs x},{\bs y}) := 
  	\kernel_s(x_1,y_1)\otimes \kernel_t(x_2,y_2)
     + \kernel_t(x_1,y_1)\otimes \kernel_s(x_2,y_2) 
\end{equation}
is the kernel of the Sobolev space of hybrid smoothness
$H_{\isomix}^{s-t,t}(\mathbb{T}^2)$. Note, however, that 
for this kernel no estimates on the approximation errors 
are known so far.
\end{remark}

Given a function $u\in H_{\mix}^p(\mathbb{T}^2)$, 
the computation of the respective kernel interpolant 
$u\mapsto {\bs P}_{\bs j}u\in {\bs V}_{\bs j} := V_{j_1}
\otimes V_{j_2}$, where 
\[
{\bs P}_{\bs j} := P_{j_1}\otimes P_{j_2}: H_{\mix}^p(\mathbb{T}^2)\to {\bs V}_{\bs j}
\]
with $P_j$ given by \eqref{eq:projectionP}, amounts 
to solving the linear system of equations
\begin{equation}\label{eq:system2D}
  ({\bs K}_{j_1}\otimes {\bs K}_{j_2}) {\bs u}_{\bs j} = {\bs f}_{\bs j}.
\end{equation}
Here ${\bs K}_{j_1}$ and ${\bs K}_{j_2}$ are the 
univariate kernel matrices defined in \eqref{eq:1Ddiscrete},
while 
\[
{\bs u}_{\bs j} = [u_{{\bs j},{\bs k}}]_{{\bs k}\in{\bs \Delta}_{\bs j}},\quad
{\bs f}_{\bs j} = [u({\bs x}_{{\bs j},{\bs k}})]_{{\bs k}\in{\bs \Delta}_{\bs j}},
\]
where
\[
{\bs \Delta}_{\bs j} := \Delta_{j_1}\times \Delta_{j_2},
\quad {\bs x}_{{\bs j},{\bs k}} := (x_{j_1,k_1},x_{j_2,k_2}).
\]
Since the system matrix is a Kronecker product of periodic 
Toeplitz matrices, the linear system \eqref{eq:system2D} 
of equations can efficiently be solved by means of the fast 
Fourier transform using well-known tensor and matrizication 
techniques found in e.g.~\cite{Hackbusch}.

Setting 
\[
  {\bs Q}_{\bs j} := Q_{j_1}\otimes Q_{j_2} 
  = (P_{j_1}-P_{j_1-1})\otimes(P_{j_2}-P_{j_2-1}):
  H_{\mix}^p(\mathbb{T}^2)\to V_{j_1}\otimes V_{j_2}
\]   
with $Q_j$ given by \eqref{eq:projectionQ}, 
and using standard tensor product arguments, we obtain 
in view of \eqref{eq:approx} the approximation property
\begin{equation}\label{eq:approx2}
 \|{\bs Q}_{\bs j} u\|_{H_{\mix}^{t_1}(\mathbb{T}^2)}\lesssim 
 2^{-(t_2-t_1)\|{\bs j}\|_1}\|u\|_{H_{\mix}^{t_2}(\mathbb{T}^2)}
 \quad\text{for all $0\le t_1\le p\le t_2\le 2p$}
\end{equation}
and in view of \eqref{eq:inverse} the inverse inequality
\begin{equation}\label{eq:inverse2}
   \|{\bs Q}_{\bs j} u\|_{H_{\mix}^{t_2}(\mathbb{T}^2)}
   \lesssim  2^{(t_2-t_1)\|{\bs j}\|_1}\|{\bs Q}_{\bs j} u\|_{H_{\mix}^{t_1}(\mathbb{T}^2)}
 \quad\text{for all $0\le t_1\le t_2 \le p$}.
\end{equation}

\begin{lemma}[Isotropic inverse estimate]
\label{lem:inverse_iso}
For $0\le s_1\le s_2\le p$ and $u\in H_{\mix}^p(\mathbb{T}^2)$, we find
\[
  \|{\bs Q}_{\bs j} u\|_{H_{\iso}^{s_2}(\mathbb{T}^2)}
  \lesssim 2^{(s_2-s_1)\|{\bf j}\|_\infty}\|{\bs Q}_{\bs j} u\|_{H_{\iso}^{s_1}(\mathbb{T}^2)}
\]
for all ${\bs j}\in\mathbb{N}_0^2$.
\end{lemma}

\begin{proof}
According to the univariate inverse estimate \eqref{eq:inverse},
we find
\begin{align*}
\|{\bs Q}_{\bs j} u\|_{H_{\iso}^{s_2}(\mathbb{T}^2)}
&\lesssim \|{\bs Q}_{\bs j} u\|_{H^{s_2}(\mathbb{T})\otimes L^2(\mathbb{T})} 
+ \|{\bs Q}_{\bs j} u\|_{L^2(\mathbb{T})\otimes H^{s_2}(\mathbb{T})}\\
&\lesssim 2^{(s_2-s_1) j_1} \|{\bs Q}_{\bs j} u\|_{H^{s_1}(\mathbb{T})\otimes L^2(\mathbb{T})} 
+ 2^{(s_2-s_1) j_2} \|{\bs Q}_{\bs j} u\|_{L^2(\mathbb{T})\otimes H^{s_1}(\mathbb{T})}\\
&\lesssim 2^{(s_2-s_1)\|{\bs j}\|_{\infty}} \|{\bs Q}_{\bs j} u\|_{H_{\iso}^{s_1}(\mathbb{T}^2)}.
\end{align*}
\end{proof}

We note that the proof applies also if we add some mixed
Sobolev regularity $t\ge 0$, i.e., we have
\begin{equation}\label{eq:inverse_shifted}
  \|{\bs Q}_{\bs j} u\|_{H_{\isomix}^{s_2,t}(\mathbb{T}^2)}
  \lesssim 2^{(s_2-s_1)\|{\bf j}\|_\infty}\|{\bs Q}_{\bs j} u\|_{H_{\isomix}^{s_1,t}(\mathbb{T}^2)}
\end{equation}
for all ${\bs j}\in\mathbb{N}_0^2$ provided that $0\le s_1\le s_2\le p-t$.
Indeed, even more general, we will proof the following result.

\begin{lemma}[General inverse estimate]
\label{lem:inverse}
For $0\le s_1\le s_2\le p$ and $0\le t_1\le t_2\le p$ such that 
$s_2+t_2\le p$ and $u\in H_{\mix}^p(\mathbb{T}^2)$, there holds
\[
  \|{\bs Q}_{\bs j} u\|_{H_{\isomix}^{s_2,t_2}(\mathbb{T}^2)}
  \lesssim 2^{(s_2-s_1)\|{\bf j}\|_\infty+(t_2-t_1)\|{\bf j}\|_1}
  \|{\bs Q}_{\bs j} u\|_{H_{\isomix}^{s_1,t_1}(\mathbb{T}^2)}
\]
for all ${\bs j}\in\mathbb{N}_0^2$.
\end{lemma}

\begin{proof}
Applying \eqref{eq:inverse_shifted} yields
\[
\|{\bs Q}_{\bs j} u\|_{H_{\isomix}^{s_2,t_2}(\mathbb{T}^2)}
\lesssim 2^{(s_2-s_1)\|{\bs j}\|_{\infty}} \|{\bs Q}_{\bs j} u\|_{H_{\isomix}^{s_1,t_2}(\mathbb{T}^2)}.
\]
In view of the inverse estimate \eqref{eq:inverse2}, we further have
\[
 \|{\bs Q}_{\bs j} u\|_{H_{\isomix}^{s_1,t_2}(\mathbb{T}^2)}
 \lesssim 2^{(t_2-t_1)\|{\bs j}\|_1} \|{\bs Q}_{\bs j} u\|_{H_{\isomix}^{s_1,t_1}(\mathbb{T}^2)}.
\]
Putting both estimates together, we obtain the desired
result.
\end{proof}

\subsection{Optimized sparse grid spaces}
We now discuss the construction of optimal sparse grid 
spaces for the approximation in the Sobolev space 
$H_{\isomix}^{s,t}(\mathbb{T}^2)$ of hybrid smoothness 
as given in \eqref{eq:GK-spaces}. To this end, define 
the \emph{optimized sparse grid space}
\begin{equation}\label{eq:hat_V}
  \widehat{\bs V}_J^\lambda := \sum_{{\bs\lambda}\in\mathcal{I}_J^{\lambda}}
  V_{j_1}\otimes V_{j_2},
\end{equation}
where the underlying index set for the levels is given by
\begin{equation}\label{eq:hat_I}
  \mathcal{I}_J^{\lambda} = \{{\bs j}\in\mathbb{N}_0^2:
  	\|{\bs j}\|_1-\lambda\|{\bs j}\|_\infty\le J(1-\lambda)\}
\end{equation}
with $\lambda\in (-\infty,1)$. Note that the choice $\lambda 
\to -\infty$ yields the classical full tensor product space while 
the choice $\lambda = 0$ results in the classical sparse grid 
space, compare Figure~\ref{fig:indexset} for an illustration.

\begin{figure}[hbt]
\begin{center}
\includegraphics[width=0.35\textwidth,height=0.32\textwidth,trim= 50 20 40 20, clip]{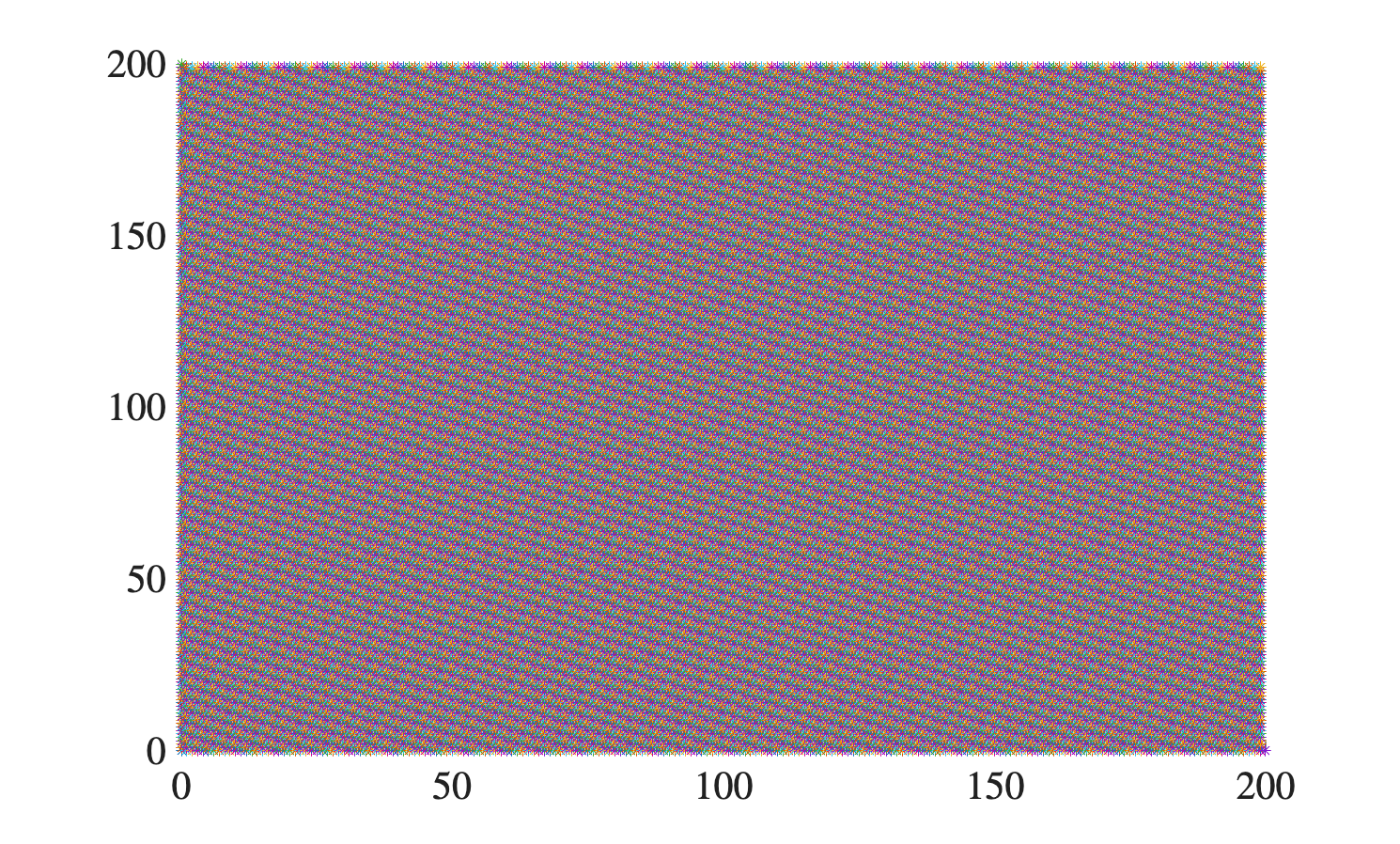}\quad
\includegraphics[width=0.35\textwidth,height=0.32\textwidth,trim= 50 20 40 20, clip]{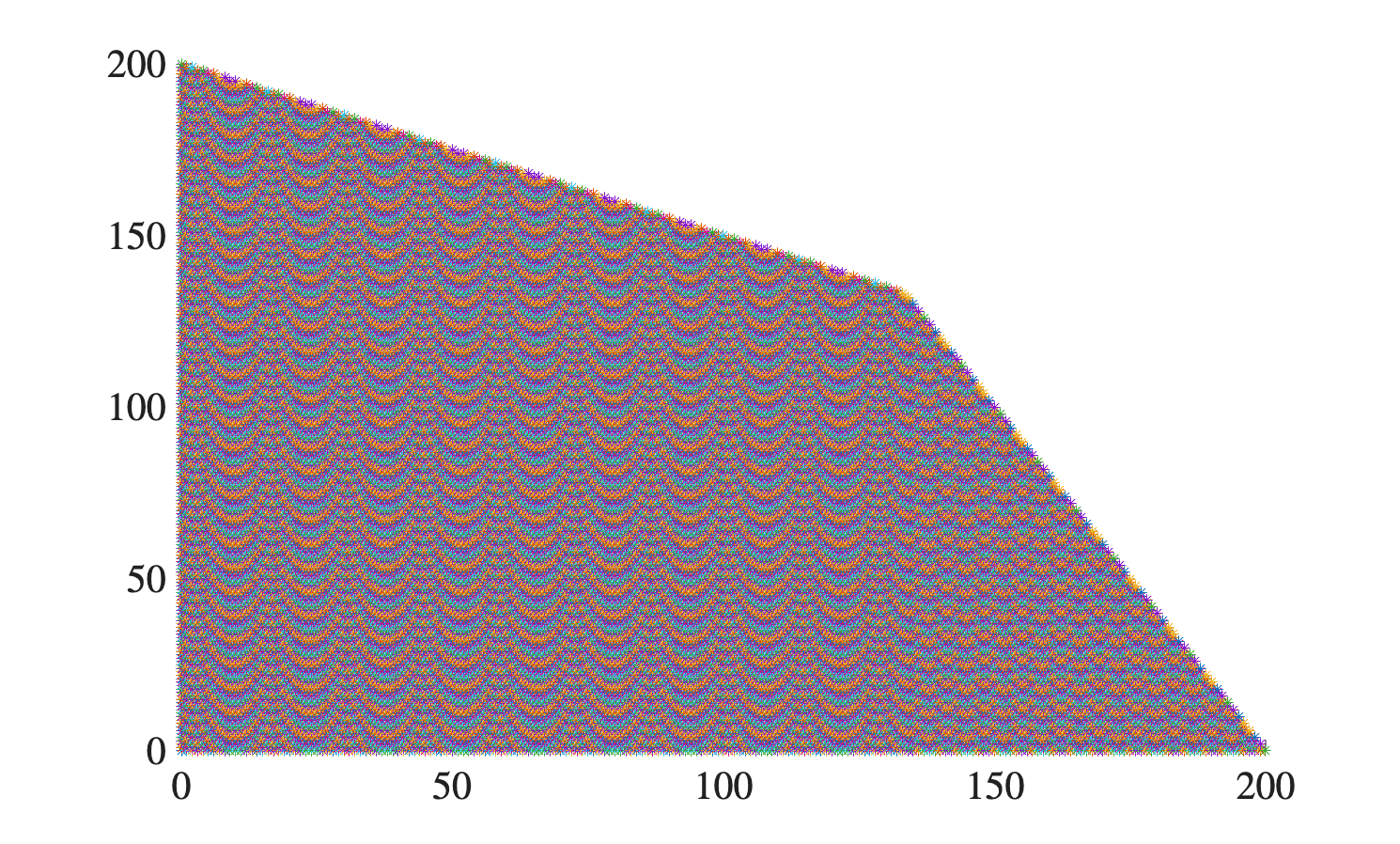}\\[2ex]
\includegraphics[width=0.35\textwidth,height=0.32\textwidth,trim= 50 20 40 20, clip]{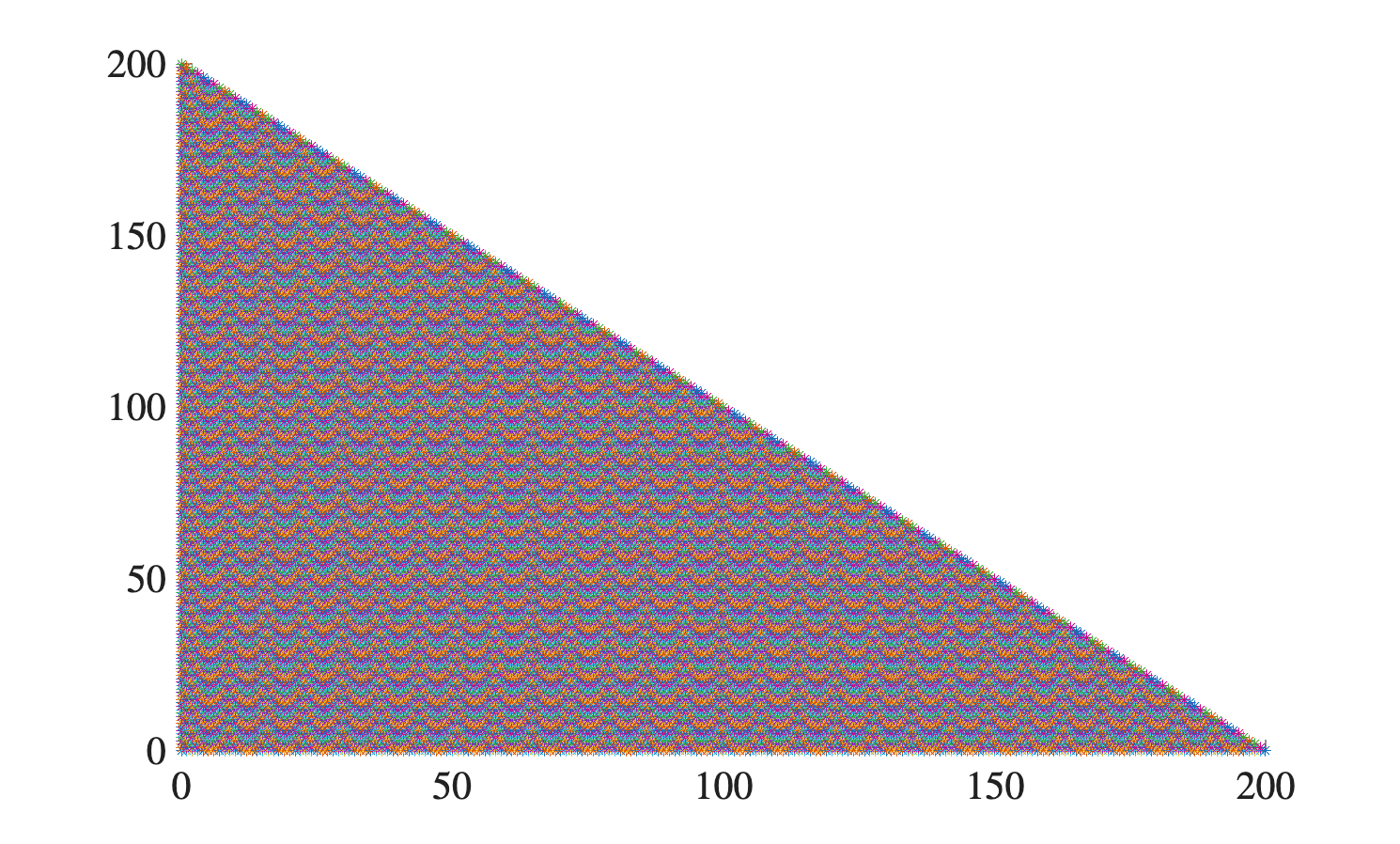}\quad
\includegraphics[width=0.35\textwidth,height=0.32\textwidth,trim= 50 20 40 20, clip]{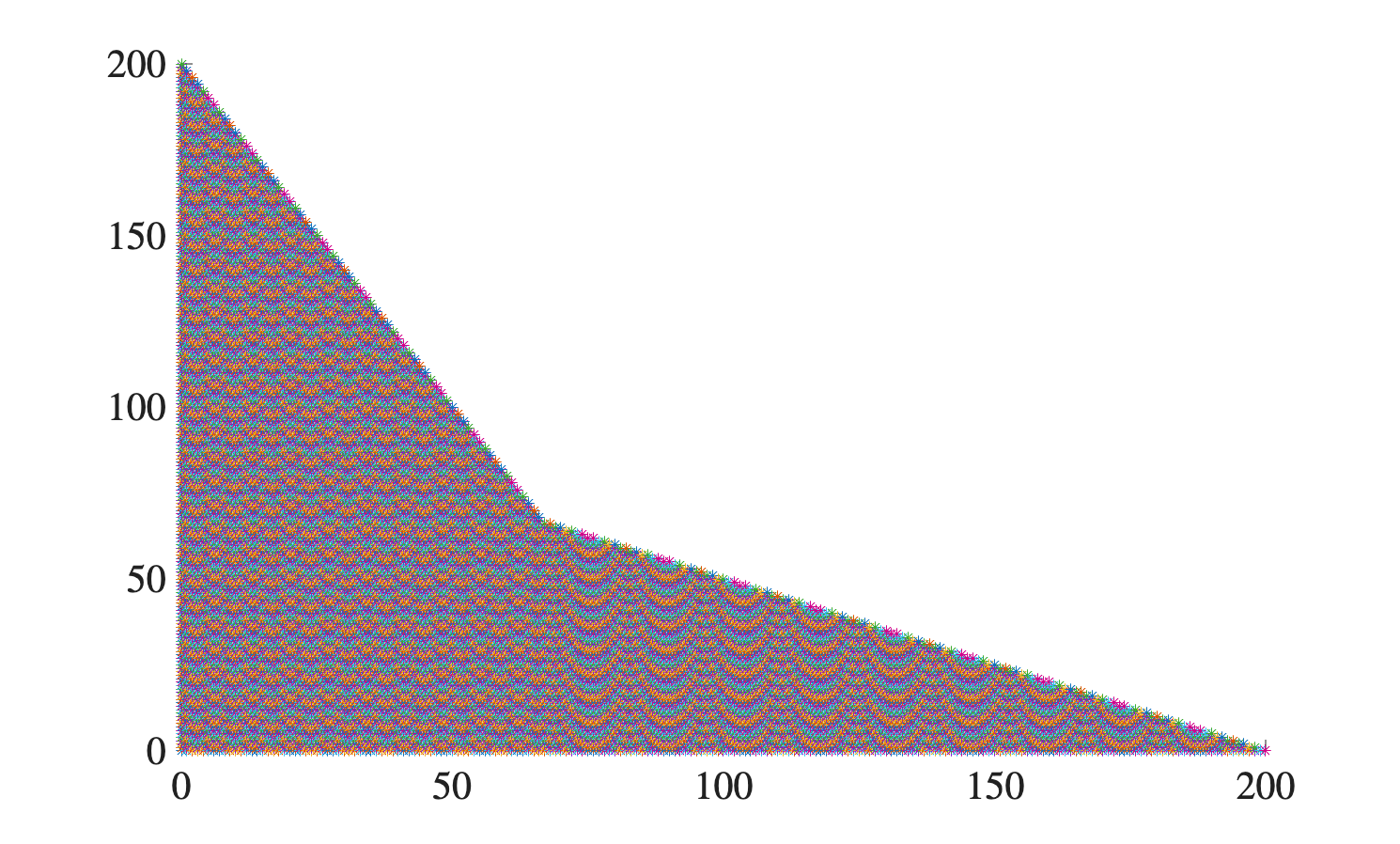}
\caption{\label{fig:indexset}
The index sets $\mathcal{I}_J^{\lambda}$ for $J=200$ and
$\lambda\to -\infty$ (top left),  $\lambda = -1$ (top right), $\lambda = 0$ 
(bottom left), and  $\lambda = 0.5$ (bottom right).}
\end{center}
\end{figure}

\begin{theorem}[Complexity]
Consider the sparse grid space $\widehat{\bs V}_J^\lambda$ given 
by \eqref{eq:hat_V} with \eqref{eq:hat_I} and $\lambda\in (-\infty,1)$. 
Then there holds 
\[
  \dim\big(\widehat{\bs V}_J^\lambda\big) 
  \lesssim \begin{cases} 
  2^J, &\text{if $\lambda > 0$},\\
  2^J J, &\text{if $\lambda = 0$},\\
  2^{2J\frac{1-\lambda}{2-\lambda}}, &\text{if $\lambda < 0$}.
  \end{cases}
\]
\end{theorem}

\begin{proof}
We assume without loss of generality that $\lambda\not=0$
as the result for the standard sparse grid is well-known, see 
\cite{BG} for example. 

We first note that the inequality in \eqref{eq:hat_I} can be
rewritten by means of $\|\bs j\|_1 = \min\{j_1,j_2\}+\max\{j_1,j_2\}$ 
and $\|\bs j\|_\infty = \max\{j_1,j_2\}$ as
\[
  \min\{j_1,j_2\} + \max\{j_1,j_2\}(1-\lambda) \le J(1-\lambda).
\]
Hence, the indices contained in $\mathcal{I}_J^{\lambda}$ 
are characterized by the inequality
\[
  \frac{\min\{j_1,j_2\}}{1-\lambda} + \max\{j_1,j_2\}\le J.
\]
The minimum and maximum switches at the level $j$
that satisfies
\begin{equation}\label{eq:CornerPoint}
\frac{j}{1-\lambda}+j = J\Rightarrow
j = J\frac{1-\lambda}{2-\lambda}.
\end{equation}
Hence, we can estimate\footnote{Note that we count here for simplicity 
certain indices twice. However, this does only enter into the generic constant.}
\begin{align*}
  \dim\big(\widehat{\bs V}_J^\lambda\big) 
  &\lesssim \sum_{j_2=0}^{J\frac{1-\lambda}{2-\lambda}} \sum_{j_1=0}^{J-\frac{j_2}{1-\lambda}} 2^{j_1+j_2}
  + \sum_{j_1=0}^{J\frac{1-\lambda}{2-\lambda}} \sum_{j_2=0}^{J-\frac{j_1}{1-\lambda}} 2^{j_1+j_2}\\
  &\lesssim\sum_{j_2=0}^{J\frac{1-\lambda}{2-\lambda}} 2^{j_2+J-\frac{j_2}{1-\lambda}}
  + \sum_{j_1=0}^{J\frac{1-\lambda}{2-\lambda}} 2^{j_1+J-\frac{j_1}{1-\lambda}}\\
  &\lesssim 2^J \Bigg[\sum_{j_2=0}^{J\frac{1-\lambda}{2-\lambda}} 2^{-j_2\frac{\lambda}{1-\lambda}}
  + \sum_{j_1=0}^{J\frac{1-\lambda}{2-\lambda}} 2^{-j_1\frac{\lambda}{1-\lambda}}\Bigg].
\end{align*}
If $\lambda > 0$, we have always negative exponents in the sums and
arrive therefore at the desired claim $\dim\big(\widehat{\bs V}_J^\lambda\big) 
\lesssim 2^J$. If $\lambda < 0$, the exponents in the sums are always 
positive which amounts to
\[
  \dim\big(\widehat{\bs V}_J^\lambda\big) 
  \lesssim 2^J 2^{-J\frac{\lambda}{2-\lambda}}
   = 2^{2J\frac{1-\lambda}{2-\lambda}}.
\]
\end{proof}

We see that the logarithm in the dimension of the optimized sparse 
grid space $\widehat{\bs V}_J^\lambda$ disappears for $\lambda > 0$,
whereas the dimension tends towards the dimension of the full tensor 
product space $V_J\otimes V_J$ for $\lambda\to-\infty$ 

\subsection{Approximation rates}
We shall next investigate the approximation power of the 
sparse grid space $\widehat{\bs V}_J^\lambda$. To this end, we
define the projection $\widehat{\bs Q}_J^\lambda: H_{\mix}^p(\mathbb{T}^2)
\to \widehat{\bs V}_J^\lambda$ onto the optimized sparse grid space 
$\widehat{\bs V}_J^\lambda$ by
\[
  \widehat{\bs Q}_J^\lambda 
  = \sum_{{\bs j}\in\mathcal{I}_J^{\lambda}} {\bs Q}_{\bs j}.
\]

\begin{figure}[htb]
\begin{center}
\begin{tikzpicture}[scale=0.6]

\fill[gray!40] (0,0) -- (6,0) -- (2,2) -- (0,6);
\draw[-] (0,6) -- (2,2); 
\draw[-] (2,2) -- (6,0);
\draw[-] (2,2) -- (2,7) node[above] {$\vdots$};
\draw[-] (2,2) -- (7,2) node[right] {$\cdots$};
    
\node at (4,4) {$C$};
\node at (5.5,1) {$A$};
\node at (1,5.5) {$B$};    
\node at (1.25,1.25) {$\mathcal{I}_J^{\lambda}$};    
      
\draw[thick,->] (0,0) -- (8,0) node[below right] {$j_1$};
\draw[thick,->] (0,0) -- (0,8) node[above left] {$j_2$};
\node[left] at (0,6) {$J$};                  
\node[left] at (0,2) {$J\frac{1-\lambda}{2-\lambda}$}; 
\draw[thick] (-0.15,6) -- (0.15,6);          
\draw[thick] (-0.15,2) -- (0.15,2);          
\node[below] at (6,0) {$J$};               
\node[below] at (2,0) {$J\frac{1-\lambda}{2-\lambda}$}; 
\draw[thick] (6,-0.15) -- (6,0.15);          
\draw[thick] (2,-0.15) -- (2,0.15);          
\end{tikzpicture}
\end{center}
\caption{\label{fig:index_sets}%
Visualization of the panelization of $\mathbb{N}_0^2$
into the index set $\mathcal{I}_J^{\lambda}$ and the index sets
which enter the sums $A$, $B$, and $C$.}
\end{figure}
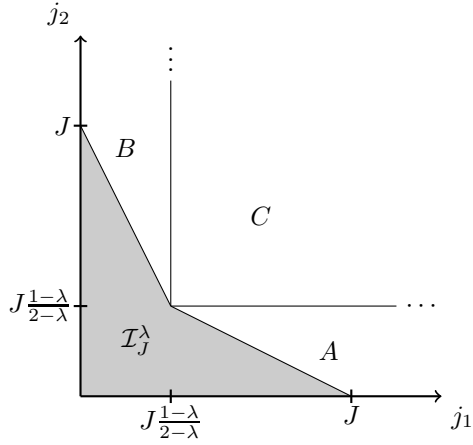

\begin{theorem}[Iso-mix-convergence]
\label{thm:iso-mix-conv}
Assume $0\le s\le p$ and $p\le t\le 2p$. Then, there holds
\[
  \big\|u-\widehat{\bs Q}_J^\lambda  u\big\|_{H_{\iso}^s(\mathbb{T}^2)}
	\lesssim 2^{-J(t-s)}\|u\|_{H_{\mix}^t(\mathbb{T}^2)}
\]
provided that $\lambda\in [0,\frac{s}{t})$. If $\lambda = 
\frac{s}{t}$, then an additional logarithmic factor appears, i.e.
\[
  \big\|u-\widehat{\bs Q}_J^\lambda  u\big\|_{H_{\iso}^s(\mathbb{T}^2)}
	\lesssim J 2^{-J(t-s)}\|u\|_{H_{\mix}^t(\mathbb{T}^2)}.
\] 
\end{theorem}

\begin{proof}
By combining the isotropic inverse estimate from 
Lemma~\ref{lem:inverse_iso} with the approximation property
\eqref{eq:approx2} we find that
\begin{align*}
  \|u-\widehat{\bs Q}_J^\lambda  u\|_{H_{\iso}^s(\mathbb{T}^2)}
  &\le\sum_{{\bs j}\not\in \mathcal{I}_J^{\lambda}}\|{\bf Q}_{\bs j}u \|_{H_{\iso}^s(\mathbb{T}^2)}\\
  &\lesssim\sum_{{\bs j}\not\in \mathcal{I}_J^{\lambda}}
  	2^{s\|{\bs j}\|_{\infty}}\|{\bf Q}_{\bs j}u\|_{L^2(\mathbb{T}^2)}\\
  &\lesssim\sum_{{\bs j}\not\in \mathcal{I}_J^{\lambda}}
  	2^{s\|{\bs j}\|_{\infty}-t\|{\bs j}\|_{1}}\|u\|_{H_{\mix}^t(\mathbb{T}^2)}.
\end{align*}
We obtain\footnote{Here and in the following, the summation 
limits are in general no natural numbers and must of course 
be rounded properly. We leave this to the reader to avoid 
cumbersome floor/ceil-notations.}
\begin{align*}
  \big\|u-\widehat{\bs Q}_J^\lambda  u\big\|_{H_{\iso}^s(\mathbb{T}^2)}
  &\lesssim \sum_{j_2=0}^{J\frac{1-\lambda}{2-\lambda}} \sum_{j_1=J-\frac{j_2}{1-\lambda}}^\infty 
  2^{s\|{\bs j}\|_{\infty}-t\|{\bs j}\|_{1}}\|u\|_{H_{\mix}^t(\mathbb{T}^2)}\\
  &+ \sum_{j_1=0}^{J\frac{1-\lambda}{2-\lambda}} \sum_{j_2=J-\frac{j_1}{1-\lambda}}^\infty
  2^{s\|{\bs j}\|_{\infty}-t\|{\bs j}\|_{1}}\|u\|_{H_{\mix}^t(\mathbb{T}^2)}\\
  &+ \sum_{j_1=J\frac{1-\lambda}{2-\lambda}}^\infty \sum_{j_2=J\frac{1-\lambda}{2-\lambda}}^\infty
  2^{s\|{\bs j}\|_{\infty}-t\|{\bs j}\|_{1}}\|u\|_{H_{\mix}^t(\mathbb{T}^2)}
  = A + B + C,
\end{align*}
where obviously $A = B$ holds, compare Figure~\ref{fig:index_sets}
for an illustration.

To estimate $B$ we use
\begin{align*}
B&\lesssim \sum_{j_1=0}^{J\frac{1-\lambda}{2-\lambda}} \sum_{j_2=J-\frac{j_1}{1-\lambda}}^\infty 
2^{(s-t)j_2-tj_1}\|u\|_{H_{\mix}^t(\mathbb{T}^2)}\\
&\lesssim \sum_{j_1=0}^{J\frac{1-\lambda}{2-\lambda}} 2^{(s-t)(J-\frac{j_1}{1-\lambda})-tj_1} 
\underbrace{\sum_{j_2=0}^\infty 2^{(s-t)j_2}}_{\lesssim 1}\|u\|_{H_{\mix}^t(\mathbb{T}^2)}\\
&\lesssim 2^{(s-t)J} \sum_{j_1=0}^{J\frac{1-\lambda}{2-\lambda}} 
2^{(t-s)\frac{j_1}{1-\lambda}-tj_1} \|u\|_{H_{\mix}^t(\mathbb{T}^2)}.
\end{align*}
If $0\le \lambda < \frac{s}{t}$, the exponent is negative and hence
the sum is uniformly bounded by a constant, leading to
\[
B\lesssim 2^{(s-t)J}\|u\|_{H_{\mix}^t(\mathbb{T}^2)}.
\]
If $\lambda = \frac{s}{t}$, the exponent in the sum is equal to zero
such that the sum is proportional to $J$ which yields an additional 
logarithmic factor, i.e.
\[
B\lesssim J 2^{(s-t)J}\|u\|_{H_{\mix}^t(\mathbb{T}^2)}.
\]

For the expression $C$ we have
\begin{align*}
  \frac{C}{2}&\lesssim \sum_{j_1=J\frac{1-\lambda}{2-\lambda}}^\infty 
  \sum_{j_2=j_1}^\infty
  2^{(s-t)j_2-tj_1}\|u\|_{H_{\mix}^t(\mathbb{T}^2)}\\
  &\lesssim\sum_{j_1=J\frac{1-\lambda}{2-\lambda}}^\infty
  2^{(s-2t)j_1} \underbrace{\sum_{j_2=0}^\infty 2^{(s-t)j_2}}_{\lesssim 1}\|u\|_{H_{\mix}^t(\mathbb{T}^2)}\\
  &\lesssim 2^{J(s-2t)\frac{1-\lambda}{2-\lambda}}
  \underbrace{\sum_{j_1=0}^\infty 2^{(s-2t)j_1}}_{\lesssim 1}\|u\|_{H_{\mix}^t(\mathbb{T}^2)}
  \lesssim 2^{J(s-t)}\|u\|_{H_{\mix}^t(\mathbb{T}^2)},
\end{align*}
where we used in the the last step that
\[
  (2t-s)\frac{1-\lambda}{2-\lambda} \ge t-s
\]
whenever $\lambda\in [0,\frac{s}{t}]$.

Putting the estimates of $A=B$ and $C$ together yields the desired result.
\end{proof}

\begin{remark}
We note that the above proof is inherently different from that in 
\cite{GK1,GK2}, which is based on wavelets, since we now cannot 
exploit norm equivalences any more for our kernel approach. As 
a consequence, we obtain the logarithmic factor $J$ for the choice 
$\lambda = \frac{s}{t}$, which is not contained in the error estimate 
of \cite{GK1,GK2}. Nevertheless, in case of $\lambda\in (0,\frac{s}{t})$, 
neither the rate of approximation nor the number $N$ of degrees of 
freedom in $\widehat{\bs V}_J^\lambda$ exhibit a logarithmic factor. 
This means that the convergence rate does not suffer from curse of 
dimension, i.e., we obtain the same rate as for the univariate kernel 
interpolation:
\[
\big\|u-\widehat{\bs Q}_J^\lambda  u\big\|_{H_{\iso}^s(\mathbb{T}^2)}
	\lesssim N^{-(t-s)}\|u\|_{H_{\mix}^t(\mathbb{T}^2)}.
\]
\end{remark}

Having this result proven, we can easily generalize it
to Sobolev spaces of hybrid regularity.

\begin{theorem}[General convergence]
\label{thm:convergence}
Assume $0\le s_2\le s_1\le p$ and $0\le t_1\le t_2\le 2p$
such that we have\footnote{The desired embedding amounts to 
the inequalities $s_1+t_1\le p\le\frac{s_2}{2}+t_2$ and 
$s_2+t_2\le 2p$.}
\[
H_{\mix}^{2p}(\mathbb{T}^2)\subset H_{\isomix}^{s_2,t_2}(\mathbb{T}^2)
\subset H_{\mix}^p(\mathbb{T}^2)\subset H_{\isomix}^{s_1,t_1}(\mathbb{T}^2)
\subset L^2(\mathbb{T}^2).
\]
Then there holds
\[
  \big\|u-\widehat{\bs Q}_J^\lambda  u\big\|_{H_{\isomix}^{s_1,t_1}(\mathbb{T}^2)}
  \lesssim 2^{-J((t_2-t_1)-(s_1-s_2))}\|u\|_{H_{\isomix}^{s_2,t_2}(\mathbb{T}^2)}
\]
provided that $\lambda\in [0,\frac{s_1-s_2}{t_2-t_1})$. If $\lambda = 
\frac{s_1-s_2}{t_2-t_1}$, an additional logarithmic factor appears, i.e.
\[
  \big\|u-\widehat{\bs Q}_J^\lambda  u\big\|_{H_{\isomix}^{s_1,t_1}(\mathbb{T}^2)}
  \lesssim J 2^{-J((t_2-t_1)-(s_1-s_2))}\|u\|_{H_{\isomix}^{s_2,t_2}(\mathbb{T}^2)}.
\]
\end{theorem}

\begin{proof}
In view of Lemma~\ref{lem:inverse}, we find
\begin{align*}
  \big\|u-\widehat{\bs Q}_J^\lambda  u\big\|_{H_{\isomix}^{s_1,t_1}(\mathbb{T}^2)}
  &\le\sum_{{\bs j}\not\in \mathcal{I}_J^{\lambda}}\|{\bf Q}_{\bs j}u \|_{H_{\isomix}^{s_1,t_1}(\mathbb{T}^2)}\\
  &\lesssim\sum_{{\bs j}\not\in \mathcal{I}_J^{\lambda}}
  	2^{(s_1-s_2)\|{\bs j}\|_{\infty}}\|{\bf Q}_{\bs j}u\|_{H_{\isomix}^{s_2,t_1}(\mathbb{T}^2)}\\
  &\lesssim\sum_{{\bs j}\not\in \mathcal{I}_J^{\lambda}}
  	2^{(s_1-s_2)\|{\bs j}\|_{\infty}-(t_2-t_1)\|{\bs j}\|_{1}}\|u\|_{H_{\isomix}^{s_2,t_2}(\mathbb{T}^2)}.
\end{align*}
Proceeding now in complete analogy to the proof of 
Theorem~\ref{thm:iso-mix-conv}, but with $s_1-s_2$ instead 
of $s$ and $t_2-t_1$ instead of $t$, yields the desired claim.
\end{proof}

\begin{remark}
The condition $s_1\ge s_2$ is essential for the present 
result, meaning that we measure the error in a space which
has a higher isotropic smoothness than the function to be
approximated. If $s_1 = s_2$, we find $\lambda = 0$, such
that the interval $[0,\frac{s_1-s_2}{t_2-t_1})$ is empty and we 
always obtain the logarithmic factor $J$ in the approximation 
error. We should further mention that the result of Theorem 
\ref{thm:convergence} could also be shown in \cite{GK1,GK2}
for $s_1<s_2$ due to the norm equivalences of wavelet bases.
This is however not possible in our situation using kernel 
interpolants. Nevertheless, $s_1<s_2$ also implies $\lambda < 0$ 
and hence results in associated optimized sparse grids which 
contain significantly more points than the regular sparse grid. 
As a consequence, the cost complexity would be higher 
than (poly-) loglinear in this situation.
\end{remark}

\section{Extension to higher dimensions}
\label{sct:arbitrary}
\subsection{Sobolev spaces of hybrid regularity}
In higher dimensions, the Sobolev spaces of dominating 
mixed derivatives are defined by
\[
  H_{\mix}^t(\mathbb{T}^d) := \bigotimes_{j=1}^d H_{\mix}^t(\mathbb{T}).
\]
This means that $H_{\mix}^t(\mathbb{T}^d)$ contains all 
periodic functions $f\in L^2(\mathbb{T}^d)$ with bounded 
derivatives $\|\partial^{\bs\alpha} f\|_{L^2(\mathbb{T}^d)}
<\infty$ for all $\|\bs\alpha\|_\infty\le t$. In contrast, 
the classical, isotropic Sobolev space
\[
  H_{\iso}^s(\mathbb{T}^d) := \bigcap_{i=1}^d 
  \Bigg[\bigg(\bigotimes_{j=1}^{i-1} L^2(\mathbb{T})\bigg)
  \otimes H^s(\mathbb{T})\otimes 
  \bigg(\bigotimes_{j=i+1}^{d} L^2(\mathbb{T})\bigg)\Bigg]
\]
contains all periodic function with bounded derivatives
$\|\partial^{\bs\alpha} f\|_{L^2(\mathbb{T}^d)}<\infty$
for all $\|\bs\alpha\|_1\le s$.

The Sobolev space of hybrid regularity can be defined
in analogy to \eqref{eq:GK-spaces} as
\[
  H_{\isomix}^{s,t}(\mathbb{T}^d) := \bigcap_{i=1}^d 
  \Bigg[\bigg(\bigotimes_{j=1}^{i-1} H^t(\mathbb{T})\bigg)
  \otimes H^{s+t}(\mathbb{T})\otimes 
  \bigg(\bigotimes_{j=i+1}^{d} H^t(\mathbb{T})\bigg)\Bigg].
\]
It consists of all periodic functions $f\in L^2(\mathbb{T}^d)$ with 
derivatives $\|\partial^{\bs\alpha} f\|_{H_{\mix}^t(\mathbb{T}^d)}
<\infty$ bounded in $H_{\mix}^t(\mathbb{T}^d)$ for all 
$\|\bs\alpha\|_1\le s$. 

Note that there holds the series of embeddings
\begin{equation}\label{eq:embedding}
 H_{\mix}^{s+\frac{t}{d}}(\mathbb{T}^d)
 \subset H_{\iso}^{s+\frac{t}{d}}(\mathbb{T}^d) 
 \subset H_{\isomix}^{s,t}(\mathbb{T}^d) \subset H_{\iso}^s(\mathbb{T}^d)
 \subset H_{\mix}^{\frac{s}{d}}(\mathbb{T}^d)
\end{equation}
for all $s,t\ge 0$ and $d\ge 1$.

\subsection{Kernel interpolation}
We now investigate kernel interpolation in $H_{\mix}^p(\mathbb{T}^d)$.
The reproducing kernel $\bs\kernel(\bs x,\bs y)$ in $H_{\mix}^p(\mathbb{T}^d)$ 
is given as the $d$-fold tensor product of the univariate kernel, which means that
\[
  \bs\kernel(\bs x,\bs y) = \prod_{k=1}^d \kernel(x_k,y_k).
\]
We define the projections ${\bs P}_{\bs j}$ and 
${\bs Q}_{\bs j}$ by
\[
  {\bs P}_{\bs j} := P_{j_1}\otimes\cdots\otimes P_{j_d},\quad
  {\bs Q}_{\bs j} := Q_{j_1}\otimes\cdots\otimes Q_{j_d}.
\]
Given a function $u\in H_{\mix}^p(\mathbb{T}^d)$, the
kernel interpolant 
\[
  {\bs P}_{\bs j} u = \sum_{{\bs k}\in\bs\Delta_{\bs j}}
  	u_{{\bs j},{\bs k}}\bs\kernel(\bs x,\bs x_{{\bs j},{\bs k}})
\]
with respect to the \emph{(full) tensor product grid} 
\[
  {\bs X}_{\bs j} := \bigtimes_{i=1}^d X_{j_i}
  	= \bigg\{{\bf x}_{{\bs j},{\bs k}} = (x_{j_1,k_1},\ldots,x_{j_d,k_d}):
  		{\bs k}\in{\bs\Delta}_{\bs j}:=\bigtimes_{\ell=1}^d \Delta_{j_\ell,k_\ell}\bigg\}
\]
is obtained by solving the system
\begin{equation}\label{eq:systemdD}
   ({\bs K}_{j_1}\otimes\cdots\otimes {\bs K}_{j_d}) {\bs u}_{\bs j} = {\bs f}_{\bs j},
\end{equation}
where the univariate kernel matrices ${\bs K}_{j_\ell} = [\kernel(x_{j_\ell,k},
x_{j_\ell,k'})]_{k,k'\in\Delta_{j_\ell}}$ are given by \eqref{eq:1Ddiscrete} and
\[
{\bs u}_{\bs j} = [u_{{\bs j},{\bs k}}]_{{\bs k}\in\bs\Delta_{\bs j}},\quad
{\bs f}_{\bs j} = [u({\bs x}_{{\bs j},{\bs k}})]_{{\bs k}\in\bs\Delta_{\bs j}}.
\]
We again note that the linear system \eqref{eq:system2D} 
of equations can efficiently be solved by using the fast 
Fourier transform in combination with well-known tensor 
and matrizication techniques found in e.g.~\cite{Hackbusch}.

Obviously, the kernel interpolant is the best approximation of
$u\in H_{\mix}^p(\mathbb{T}^d)$ in the subspace
\[
  {\bs V}_{\bs j} = \bigotimes_{i=1}^d V_{j_i} 
  = \spn\{\kernel(\cdot,{\bs x}):
  	{\bs x}\in {\bs X}_{\bs j}\}\subset H_{\mix}^p(\mathbb{T}^d)
\]
with respect to the $H_{\mix}^p(\mathbb{T}^d)$-norm.

\subsection{Bernstein and Jackson inequalities}
We find the following straightforward extension of 
Lemma~\ref{lem:inverse} to the $d$-dimensional setting:

\begin{lemma}[General inverse estimate]
For $0\le s_1\le s_2\le p$ and $0\le t_1\le t_2\le p$ such that 
$s_2+t_2\le p$ and $u\in H_{\mix}^p(\mathbb{T}^d)$, there holds
\[
  \|{\bs Q}_{\bs j} u\|_{H_{\isomix}^{s_2,t_2}(\mathbb{T}^d)}
  \lesssim 2^{(s_2-s_1)\|{\bf j}\|_\infty+(t_2-t_1)\|{\bf j}\|_1}
  \|{\bs Q}_{\bs j} u\|_{H_{\isomix}^{s_1,t_1}(\mathbb{T}^d)}
\]
for all ${\bs j}\in\mathbb{N}_0^d$.
\end{lemma}

Now, we introduce for $\lambda\in (-\infty,1)$ the index set
\[
  \mathcal{I}_J^{\lambda} = \{{\bs j}\in\mathbb{N}_0^d:
  	\|{\bs j}\|_1-\lambda\|{\bs j}\|_\infty\le J(1-\lambda)\}.
\]
This yields the \emph{optimized sparse grid space}
\[
  \widehat{\bs V}_J^\lambda := \sum_{{\bs j}\in\mathcal{I}_J^{\lambda}}
  {\bs V}_{\bs j} = \sum_{{\bs j}\in\mathcal{I}_J^{\lambda}}\bigotimes_{i=1}^d V_{j_i},
\]
which has the dimension
\[
  \dim\big(\widehat{\bs V}_J^\lambda\big) 
  \lesssim \begin{cases} 
  2^J, &\text{if $\lambda > 0$},\\
  2^J J^{d-1}, &\text{if $\lambda = 0$},\\
  2^{Jd\frac{1-\lambda}{d-\lambda}}, &\text{if $\lambda < 0$}.
  \end{cases}
\]
The last case is seen as follows: If we are looking for the
largest level ${\bf j}\in\mathbb{N}_0^d$ with $j_1=\cdots = j_d =: j$
such that the index ${\bs j}$ is still contained in $\mathcal{I}_J^{\lambda}$,
we find
\[
  d j - \lambda j = J(1-\lambda) \Rightarrow j = J\frac{1-\lambda}{d-\lambda}
\]
in analogy to \eqref{eq:CornerPoint}. This is the largest full 
tensor product space ${\bs V}_{\bs j} = \bigotimes_{j=1}^d V_j$ 
that is contained in $\widehat{\bs V}_J^\lambda$. It has 
\[
  \dim({\bs V}_{\bs j}) = \dim(V_j)^d = 2^{Jd\frac{1-\lambda}{d-\lambda}}
\]
degrees of freedom and gives the bound for
$\dim\big(\widehat{\bs V}_J^\lambda\big)$ in the case 
$\lambda < 0$, compare also \cite{GK1,GK2}.

Finally, the generalization of Theorem~\ref{thm:convergence}
to the $d$-dimensional setting reads as follows, compare
also \cite{GK1,GK2}:

\begin{theorem}[General convergence]
Assume $0\le s_2\le s_1\le p$ and $0\le t_1\le t_2\le 2p$
such that\footnote{In view of \eqref{eq:embedding}, the 
desired embedding amounts to the inequalities $s_1+t_1\le p
\le\frac{s_2}{d}+t_2$ and $s_2+t_2\le 2p$.}
\begin{equation}\label{eq:inclusions}
H_{\mix}^{2p}(\mathbb{T}^d)\subset H_{\isomix}^{s_2,t_2}(\mathbb{T}^d)
\subset H_{\mix}^p(\mathbb{T}^d)\subset H_{\isomix}^{s_1,t_1}(\mathbb{T}^d)
\subset L^2(\mathbb{T}^d).
\end{equation}
Then there holds
\[
  \big\|u-\widehat{\bs Q}_J^\lambda  u\big\|_{H_{\isomix}^{s_1,t_1}(\mathbb{T}^d)}
  \lesssim 2^{-J((t_2-t_1)-(s_1-s_2))}\|u\|_{H_{\isomix}^{s_2,t_2}(\mathbb{T}^d)}
\]
provided that $\lambda\in [0,\frac{s_1-s_2}{t_2-t_1})$. If $\lambda = 
\frac{s_1-s_2}{t_2-t_1}$, an additional polylogarithmic factor appears, i.e.
\[
  \big\|u-\widehat{\bs Q}_J^\lambda  u\big\|_{H_{\isomix}^{s_1,t_1}(\mathbb{T}^d)}
  \lesssim J^{d-1} 2^{-J((t_2-t_1)-(s_1-s_2))}\|u\|_{H_{\isomix}^{s_2,t_2}(\mathbb{T}^d)}.
\]
\end{theorem}

\begin{proof}
The proof is along the lines of the proof of Theorem~\ref{thm:iso-mix-conv}.
One has
\begin{align*}
  \big\|u-\widehat{\bs Q}_J^\lambda  u\big\|_{H_{\isomix}^{s_1,t_1}(\mathbb{T}^2)}
  &\le\sum_{{\bs j}\not\in \mathcal{I}_J^{\lambda}}\|{\bf Q}_{\bs j}u \|_{H_{\isomix}^{s_1,t_1}(\mathbb{T}^2)}\\
  &\lesssim\sum_{{\bs j}\not\in \mathcal{I}_J^{\lambda}}
  	2^{(s_1-s_2)\|{\bs j}\|_{\infty}}\|{\bf Q}_{\bs j}u\|_{H_{\isomix}^{s_2,t_1}(\mathbb{T}^2)}\\
  &\lesssim\sum_{{\bs j}\not\in \mathcal{I}_J^{\lambda}}
  	2^{(s_1-s_2)\|{\bs j}\|_{\infty}-(t_2-t_1)\|{\bs j}\|_{1}}\|u\|_{H_{\isomix}^{s_2,t_2}(\mathbb{T}^2)}.
\end{align*}
We find
\[
 \{{\bs j}\in\mathbb{N}_0^d:\|{\bs j}\|_1-\lambda \|{\bs j}\|_{\infty} > J(1-\lambda)\}
 = \bigcup_{k=1}^d A_k
\]
with
\[
  A_k := \Bigg\{{\bs j}\in\mathbb{N}_0^d: j_k = \|{\bs j}\|_\infty \wedge
 	\sum_{\begin{smallmatrix} i=1 \\ i\not= k\end{smallmatrix}}^d 
  	j_i+(1-\lambda) j_k > J(1-\lambda)\Bigg\}.
\]
Hence,
\begin{align*}
  \big\|u-\widehat{\bs Q}_J^\lambda  u\big\|_{H_{\isomix}^{s_1,t_1}(\mathbb{T}^2)}
  &\lesssim \sum_{k=1}^d \sum_{{\bs j}\in A_k}
  2^{-(t_2-t_1)\sum_{i\not= k} j_i} 2^{((s_1-s_2)-(t_2-t_1))j_k}\|u\|_{H_{\isomix}^{s_2,t_2}(\mathbb{T}^2)}.
\end{align*}
Since there holds
\[
  A_k \subset B_k := \Bigg\{{\bs j}\in\mathbb{N}_0^d:
 	\sum_{\begin{smallmatrix} i=1 \\ i\not= k\end{smallmatrix}}^d 
  	j_i+(1-\lambda) j_k > J(1-\lambda)\Bigg\},
\]
we can further estimate 
\begin{align*}
  \big\|u-\widehat{\bs Q}_J^\lambda  u\big\|_{H_{\isomix}^{s_1,t_1}(\mathbb{T}^2)}
  &\lesssim \sum_{k=1}^d \sum_{{\bs j}\in B_k}
  2^{-(t_2-t_1)\sum_{i\not= k} j_i} 2^{((s_1-s_2)-(t_2-t_1))j_k}\|u\|_{H_{\isomix}^{s_2,t_2}(\mathbb{T}^2)}.
\end{align*}
The error contributions in each of the simplicial index sets $B_k$ 
corresponds to the error of a generalized sparse grid with weights
$\alpha_i = \frac{1}{1-\lambda}$ for $i\not= k$ and $\alpha_k = 1$,
compare \cite{SG}. Thus, the $k$-th error contributions can be 
estimated by using \cite{SG}, leading to the desired bound. In 
particular, the polylogarithmic factor $J^{d-1}$ appears only in 
the case when the error contributions are identical along the 
boundary of the index set $\mathcal{I}_J^\lambda$, which is 
only the case for $\lambda = \frac{s_1-s_2}{t_2-t_1}$.
\end{proof}

\subsection{Computation of the sparse grid kernel interpolant}
We should finally comment on the computation of the kernel interpolant 
$\widehat{\bs Q}_J^\lambda  u$. We may employ the \emph{sparse grid 
combination technique} as introduced in \cite{GSZ,Smolyak}.
To this end, we use the identity
\begin{equation}\label{eq:combiformel}
\widehat{\bs Q}_J^{\lambda} = \sum_{{\bs j}\in\mathcal{I}_J^{\lambda}}
c_{\bs j} {\bs P}_{\bs j},
\quad\text{where }
c_{\bs j}:=\sum_{\begin{smallmatrix}{\bs j}'\in\{0,1\}^d:\\
{\bs j}+{\bs j}'\in\mathcal{I}_J^{\lambda}\end{smallmatrix}}
(-1)^{\|{\bs j}'\|_1},
\end{equation}
compare \cite{D1,D2}. Hence, the sought sparse grid kernel 
interpolant $\widehat{\bs Q}_J^{\lambda} u$ is composed by the
tensor product kernel interpolants $u_{\bs j} := {\bs P}_{\bs j} u$ 
from the different full tensor product spaces ${\bs V}_{\bs j}$ with
$c_{\bs j}\not= 0$. Each of the tensor product kernel interpolants 
$u_{\bs j}$ can now be computed (completely in parallel) by solving 
the linear system \eqref{eq:systemdD} of equations. We emphasize 
that the combination technique does not introduce an additional 
consistency error for the problem under consideration, 
see \cite{GHM} for a proof.

\section{Conclusion}\label{sec:conclusio}
In the present article, we have shown that the approximation 
rate for kernel interpolation in $H_{\mix}^p(\mathbb{T}^d)$ with
respect to optimized sparse grids, i.e.
\begin{equation}\label{eq:final}
  \big\|u-\widehat{\bs Q}_J^\lambda  u\big\|_{H_{\isomix}^{s_1,t_1}(\mathbb{T}^d)}
  \lesssim 2^{-J((t_2-t_1)-(s_1-s_2))}\|u\|_{H_{\isomix}^{s_2,t_2}(\mathbb{T}^d)},
\end{equation}
is dimension independent for the choice $0<\lambda
<\frac{s_1-s_2}{t_2-t_1}$ whenever $0\le s_2\le s_1\le p$ 
and $0\le t_1\le t_2\le 2p$ such that \eqref{eq:inclusions} 
holds and $u\in H_{\mix}^p(\mathbb{T}^d)$ is sufficiently 
smooth. Since likewise the dimension of the optimized 
sparse grid space is dimension independent, that is
$N_J := \dim\big(\widehat{\bs V}_J^\lambda\big) \sim 2^J$
for $\lambda > 0$, we arrive at \emph{dimension independent}
cost complexity
\begin{equation}\label{eq:final-epsilon}
  \big\|u-\widehat{\bs Q}_J^\lambda  u\big\|_{H_{\isomix}^{s_1,t_1}(\mathbb{T}^d)}
  \lesssim N_J^{-1/((t_2-t_1)-(s_1-s_2))}\|u\|_{H_{\isomix}^{s_2,t_2}(\mathbb{T}^d)}
\end{equation}
for $0<\lambda<\frac{s_1-s_2}{t_2-t_1}$ whenever 
$0\le s_2\le s_1\le p$ and $0\le t_1\le t_2\le 2p$ such 
that \eqref{eq:inclusions} holds. Nevertheless, we like 
to emphasize that the generic constants which are 
involved in \eqref{eq:final} and \eqref{eq:final-epsilon},
respectively, still depend (exponentially) on the 
dimension $d$. 

The result \eqref{eq:final} carries over straighforwardly to 
\emph{quasi-uniform} point sets
\[
  X_0 \subset X_1 \subset X_2\subset \cdots\subset \mathbb{T}
\]
instead of equidistant point sets, where the cardinality 
of the point sets 
\[
X_j := \{x_{j,k}:k\in\Delta_j\}
\]
satisfies $|\Delta_j|\sim 2^j$. The notion quasi-uniform means in the
present context that the \emph{fill distance} satisfies
\[
h_j:= \sup_{x\in\mathbb{T}} \min_{x_k\in X_j} \|x - x_k\|_2 \lesssim 2^{-j}
\]
while the \emph{separation radius} satisfies
\[
q_j := \min_{\begin{smallmatrix} x_k,x_\ell\in X_j:
\\ x_k \neq x_\ell\end{smallmatrix}} \frac{1}{2}\| x_k - x_\ell \|_2\gtrsim 2^{-j}.
\]

Our result applies moreover one-to-one to general product 
domains and the non-periodic case as described in \cite{GHM}, if the 
\emph{doubling trick} from \cite{Schaback1,Sloan} is applicable. Thanks 
to this trick we are able to exploit extra smoothness if present for the function 
to be approximated. Nonetheless, if the doubling trick does not apply, 
we still have an estimate of the form 
\[
  \big\|u-\widehat{\bs Q}_J^\lambda  u\big\|_{H_{\isomix}^{s,t}(\mathbb{T}^d)}
  \lesssim 2^{-J((p-t)-s)}\|u\|_{H_{\mix}^{p}(\mathbb{T}^d)}
\]
whenever $0<s$, $0\le t$, and $s+\frac{t}{d}\le p$. 
Hence, in general, no logarithm appears in the approximation 
rate for kernel interpolation if we measure the error with respect 
to an isotropic Sobolev space being different from $L^2(\mathbb{T}^d)$ 
for the choice $0<\lambda<\frac{s}{t}$.

In case of the Schr\"odinger equation, one has the product of 
three-dimensional one-particle spaces instead of just one-dimensional 
spaces like in our article here and $s_1 = 1$, $t_1 = 0$, and $s_2 = 1$, 
while $\frac{1}{2}\le t_2\le 1$ depends on the particular symmetry behaviour 
of the wavefunctions, see~\cite{Meng,Harry2} for the details. This means 
that the usual sparse grid space (i.e., the space $\widehat{\bs V}_J^\lambda$ 
with $\lambda = 0$, now on three-dimensional particle spaces) would be
optimal, which results in (poly-) logarithmic factors in the number of 
degrees of freedom as well as in the rate of approximation. Note that 
the antisymmetry of the wavefunctions can be built into the kernel 
interpolant in accordance with \cite{antisymmtry}. 

\section*{Acknowledgments}
Michael Griebel was supported by the \emph{Hausdorff Center for Mathematics} 
(HCM) in Bonn, funded by the Deutsche Forschungsgemeinschaft (DFG, German 
Research Foundation) under Germany's Excellence Strategy -- EXC-2047/1 -- 
390685813 and by the CRC 1720 \emph{Analysis of criticality: From complex phenomena
to models and estimates} -- 539309657 of the Deutsche Forschungsgemeinschaft. Helmut Harbrecht was 
funded in parts by the Swiss National Science Foundation (SNSF) through the 
Vietnamese-Swiss Joint Research Project IZVSZ2\_229568.

\end{document}